\documentclass{amsart}

\makeindex

\usepackage{amsfonts,graphics,amsmath,amsthm,amsfonts,amscd,amssymb,amsmath,latexsym,multicol,euscript}
\usepackage{epsfig,url}
\usepackage{flafter}
\usepackage{hyperref}
\usepackage[all,cmtip,line]{xy}
\usepackage{array}
\usepackage[english]{babel}
\usepackage{overpic}
\usepackage{subfig}
\usepackage{multirow}

\usepackage{wrapfig}
\usepackage{enumitem}

\theoremstyle{definition}

\theoremstyle{theorem}
\newtheorem{theoremalpha}{Theorem}
\theoremstyle{corollary}

\newtheorem{corollaryalpha}[theoremalpha]{Corollary}

\newtheorem{theorem}{Theorem}[section]
\newtheorem{lemma}[theorem]{Lemma}

\theoremstyle{corollary}
\newtheorem{corollary}[theorem]{Corollary}

\theoremstyle{definition}
\newtheorem{definition}[theorem]{Definition}

\newtheorem{example}[theorem]{Example}

\newtheorem{remark}[theorem]{Remark}

\numberwithin{equation}{section}

\newcommand\C{\mathbb{C}}
\newcommand\R{\mathbb{R}}
\newcommand\Q{\mathbb{Q}}
\newcommand\Z{\mathbb{Z}}
\renewcommand\P{\mathbb{P}}

\newcommand\eps{\varepsilon}

\newcommand{\mc}{\mathcal}

\newcommand{\ol}[1]{\overline{#1}}

\DeclareMathOperator{\codim}{codim}

\DeclareMathOperator{\ord}{ord}
\DeclareMathOperator{\mult}{mult}

\DeclareMathOperator{\Supp}{Supp}

\DeclareMathOperator{\Cent}{Cent}
\DeclareMathOperator{\vol}{vol}

\DeclareMathOperator{\Eff}{Eff}

\DeclareMathOperator{\Pic}{Pic}

\newcommand{\bm}{\mathbf B_-}  
\newcommand{\bp}{\mathbf B_+}  

\newcommand{\okbd}{\Delta}
\newcommand{\okval}{\Delta^{\text{val}}}
\newcommand{\oklim}{\Delta^{\lim}}

\newcommand{\Obody}{\text{Okounkov body }}

\begin{document}

\title[Asymptotic base loci via Okounkov bodies]{Asymptotic base loci via Okounkov bodies}

\author{Sung Rak Choi}
\address{Department of Mathematics, Yonsei University, Seoul, Korea}
\email{sungrakc@yonsei.ac.kr}

\author{Yoonsuk Hyun}
\address{Samsung Advanced Institute of Technology, Suwon, Korea}
\email{yoonsuk.hyun@samsung.com}

\author{Jinhyung Park}
\address{School of Mathematics, Korea Institute for Advanced Study, Seoul, Korea}
\email{parkjh13@kias.re.kr}

\author{Joonyeong Won}
\address{Center for Geometry and Physics, Institute for Basic Science, Pohang, Korea}
\email{leonwon@ibs.re.kr}

\date{\today}
\thanks{S. Choi and J. Park were partially supported by the National Research Foundation of Korea (NRF) grant funded by the Korea government (Ministry of Science and ICT) (No. NRF-2016R1C1B2011446).
J. Won was partially supported by IBS-R003-D1, Institute for Basic Science in Korea. }

\begin{abstract}
An Okounkov body is a convex subset of Euclidean space associated to a divisor on a smooth projective variety with respect to an admissible flag.
In this paper, we recover the asymptotic base loci from the Okounkov bodies
by studying various asymptotic invariants such as the asymptotic valuations and the moving Seshadri constants.
Consequently, we obtain the nefness and ampleness criteria of divisors in terms of the Okounkov bodies.
Furthermore, we compute the divisorial Zariski decomposition by the Okounkov bodies, and find upper and lower bounds for moving Seshadri constants given by the size of simplexes contained in the Okounkov bodies.
\end{abstract}

\maketitle

\tableofcontents

\section{Introduction}

It is a fundamental problem to understand the geometry of linear series or divisors on a variety in algebraic geometry. Since the introduction and treatment of the Okounkov bodies associated to big divisors by
Lazarsfeld-Musta\c{t}\u{a} (\cite{lm-nobody}) and Kaveh-Khovanskii (\cite{KK}) motivated by earlier works by Okounkov (\cite{O1}, \cite{O2}), there have been considerable attempts to extract various properties of divisors from the Okounkov bodies.
Let $D$ be a divisor on a smooth projective variety $X$ of dimension $n$.
Then the Okounkov body $\okbd_{Y_\bullet}(D)$ is a convex subset of the Euclidean space $\R^n$ associated to $D$ with respect to an admissible flag $Y_\bullet$.
In \cite{CHPW}, we defined and studied the valuative Okounkov body $\okval_{Y_\bullet}(D)$ and the limiting Okounkov body $\oklim_{Y_\bullet}(D)$ of a pseudoeffective divisor $D$ with respect to an admissible flag $Y_\bullet$.
They are also convex bodies in the Euclidean space $\R^n$ which coincide with the classical Okounkov body $\okbd_{Y_\bullet}(D)$ if $D$ is big.
For more details on Okounkov bodies, see Section \ref{oksec}.

It was shown that two pseudoeffective divisors are numerically equivalent to each other if and only if the associated limiting Okounkov bodies with respect to all admissible flags coincide (\cite[Proposition 4.1]{lm-nobody}, \cite[Theorem A]{jow}, \cite[Theorem C]{CHPW}).
We also refer to \cite{Roe} for a recent study on the local positivity and the local numerical equivalence using Okounkov bodies in the surface case.
In principle, every numerical property of pseudoeffective divisors can be encoded in the associated limiting Okounkov bodies with respect to all admissible flags.
On the other hand, the valuative Okounkov bodies are not numerical in nature and the main results of this paper do not hold for such bodies (see \cite[Remark 4.10]{CPW}).

One of the most important numerical properties of pseudoeffective divisors is the asymptotic base loci.
The principal aim of this paper is to study how to extract asymptotic base loci, more precisely, the restricted base locus $\bm(D)$ and the augmented base locus $\bp(D)$, from the limiting Okounkov bodies of a pseudoeffective divisor $D$. See Subsection \ref{asybasubsec} for definitions of asymptotic base loci.

The following is the first main result of this paper on the restricted base loci.

\begin{theoremalpha}[=Theorem \ref{thrm-B-criterion}]\label{main1}
Let $D$ be a pseudoeffective divisor on a smooth projective variety $X$ of dimension $n$.
Then the following are equivalent.
\begin{enumerate}
 \item[$(1)$] $x\in\bm(D)$.

 \item[$(2)$] For any admissible flag $Y_\bullet$ centered at $x$, the origin of $\R^n$ is not contained in $\oklim_{Y_\bullet}(D)$.

 \item[$(3)$] For some admissible flag $Y_\bullet$ centered at $x$, the origin of $\R^n$ is not contained in $\oklim_{Y_\bullet}(D)$.
\end{enumerate}
\end{theoremalpha}

Note that a pseudoeffective divisor $D$ is nef if and only if $\bm(D)=\emptyset$. Thus we immediately obtain the following nefness criterion of divisors.

\begin{corollaryalpha}[=Corollary \ref{cor-nefness}]\label{maincor1}
Let $D$ be a divisor on a smooth projective variety $X$ of dimension $n$.
Then the following are equivalent.
\begin{enumerate}
 \item[$(1)$] $D$ is nef.

 \item[$(2)$] For any admissible flag $Y_\bullet$, the origin of $\R^n$ is contained in $\oklim_{Y_\bullet}(D)$.

 \item[$(3)$] For any point $x \in X$, there exists an admissible flag $Y_\bullet$ centered at $x$ such that the origin of $\R^n$ is contained in $\oklim_{Y_\bullet}(D)$.
\end{enumerate}
\end{corollaryalpha}

To prove Theorem \ref{main1}, we use the asymptotic valuation at a given pseudoeffective divisor (see Subsection \ref{asyvalsubsec} for the definition).
Furthermore, we recover the divisorial components of $\bm(D)$ from the limiting Okounkov bodies, thereby obtaining the movability criterion of divisors (see Theorem \ref{thrm-movability}). We also compute the divisorial Zariski decomposition of a pseudoeffective divisor (see Section \ref{divzdsec}).

Next we prove the analogous results for the augmented base locus.
We define $U_{\geq 0}:=U \cap \R^n_{\geq 0}$ where $U$ is a small open neighborhood of the origin of $\R^n$.

\begin{theoremalpha}[=Theorem \ref{thrm-B+criterion}]\label{main2}
Let $D$ be a pseudoeffective divisor on a smooth projective variety $X$ of dimension $n$.
Then the following are equivalent.
\begin{enumerate}
\item[$(1)$] $x \in \bp(D)$.

\item[$(2)$] For any admissible flag $Y_\bullet$ centered at $x$, the subset $U_{\geq0}$ of $\R^n_{\geq 0}$ is not contained in $\oklim_{Y_\bullet}(D)$ for any small open neighborhood $U$ of the origin of $\R^n$.

\item[$(3)$] For some admissible flag $Y_\bullet$ centered at $x$, the subset $U_{\geq0}$ of $\R^n_{\geq 0}$ is not contained in $\oklim_{Y_\bullet}(D)$ for any small open neighborhood $U$ of the origin of $\R^n$.
\end{enumerate}
\end{theoremalpha}

Note that a divisor $D$ is ample if and only if $\bp(D)=\emptyset$. Thus we immediately obtain the following ampleness criterion of divisors.

\begin{corollaryalpha}[=Corollary \ref{cor-ampleness}]\label{maincor2}
Let $D$ be a divisor on a smooth projective variety $X$ of
dimension $n$. Then the following are equivalent.
\begin{enumerate}
 \item[$(1)$] $D$ is ample.

 \item[$(2)$] For any admissible flag $Y_\bullet$, the subset $U_{\geq0}$ of $\R^n_{\geq 0}$ is contained in $\okbd_{Y_\bullet}(D)$ for some small open neighborhood $U$ of the origin of $\R^n$.

 \item[$(3)$] For any point $x \in X$, there exists an admissible flag $Y_\bullet$ centered at $x$ such that the subset $U_{\geq0}$ of $\R^n_{\geq 0}$ is contained in $\okbd_{Y_\bullet}(D)$ for some small open neighborhood $U$ of the origin of $\R^n$.
\end{enumerate}
\end{corollaryalpha}

To prove Theorem \ref{main2}, we use some results on the moving Seshadri constants  (see Subsection \ref{movsessubsec} for the definition), slices of Okounkov bodies (\cite[Theorem 1.1]{CPW2}), and a version of the Fujita approximation (\cite[Proposition 3.7]{lehmann-nu}).

We can also give both lower and upper bounds for the moving Seshadri constants of pseudoeffective divisors by analyzing the structure of the limiting Okounkov bodies.
A \emph{simplex in $\R^n_{\geq0}$ of length $\lambda=(\lambda_1, \cdots,\lambda_n)$} is a convex set defined as
$$\blacktriangle_\lambda:=
\left\{(x_1,\cdots,x_n)\in\R^n_{\geq0}\left| \; \frac{x_1}{\lambda_1}+\cdots+\frac{x_n}{\lambda_n}\leq 1\right.\right\}
$$
where $\lambda_i \geq 0$ ($1\leq i\leq n$) are nonnegative real numbers and we let $x_i=0$ for $i$ such that $\lambda_i=0$.
If $x \not\in \bm(D)$, then the origin of $\R^n$ is contained in $\oklim_{Y_\bullet}(D)$ for any admissible flag $Y_\bullet$ by Theorem \ref{main1}. Thus $\blacktriangle_{\lambda}\subseteq\oklim_{Y_\bullet}(D)$ for some length $\lambda$. For $x\not\in\bm(D)$ and an admissible flag $Y_\bullet$ centered at $x$, we consider the \emph{maximal sub-simplex} $\blacktriangle_{\max}$ contained in $\oklim_{Y_\bullet}(D)$.
If the simplex $\blacktriangle_{\max}$ has length $(\lambda_1,\cdots,\lambda_n)$, then we denote
$\lambda_i(D;x,Y_\bullet):=\lambda_i$ and $\lambda_{\min}(D;x,Y_\bullet):=\min_i\{\lambda_i\;|\;i=1,\cdots,n\}$.
If $x\in\bm(D)$ so that the origin is not contained in $\oklim_{Y_\bullet}(D)$, then we define $\blacktriangle_{\max}$ to be the empty set and $\lambda_i=0$ for every $i$.

\begin{theoremalpha}[=Theorem \ref{movsesthm}]\label{main3}
Let $D$ be a pseudoeffective divisor on a smooth projective variety $X$ of dimension $n$, and $x$ be a point on $X$. Then we have
$$
\sup_{Y_\bullet}\{\lambda_{\min}(D;x,Y_\bullet)\} \leq \eps(||D||;x) \leq \inf_{Y_\bullet}\{\lambda_n(D;x,Y_\bullet)\}
$$
where $\sup$ and $\inf$ are taken over all the admissible flags $Y_\bullet$ centered at $x$.
\end{theoremalpha}

We prove Theorem \ref{main3} by basically reducing our statement to the case of nef divisors (Theorem \ref{sesthm}) using a version of the Fujita approximation (\cite[Proposition 3.7]{lehmann-nu}).

The equalities in Theorem \ref{main3} hold for some cases (see Example \ref{equality}), and there are  nontrivial examples that the Seshadri constant is computed by the Okounkov bodies (see Remark \ref{gromov}). Note also that both inequalities in Theorem \ref{main3} can be strict in general (see Example \ref{strineqrem}). It is too much to expect to obtain the exact values of the moving Seshadri constants by only considering the Okounkov bodies on $X$.
On the other hand, one can obtain the exact values by using the infinitesimal Okounkov bodies (see \cite[Remark 5.5]{lm-nobody}, \cite[Theorem C]{AV-loc pos3}). However, computing the infinitesimal Okounkov bodies is quite difficult in general. Moreover, it is already very interesting to give some bounds for moving Seshadri constants using the Okounkov bodies only (cf. \cite{I}, \cite{AV-loc pos}).

Our main results generalize main results in \cite{AV-loc pos} into higher dimensions.
We remark that K\"{u}ronya and Lozovanu also independently obtained Theorem \ref{main1} and Corollary \ref{maincor1} in \cite{AV-loc pos2} when the divisor $D$ is big. They also showed Theorem \ref{main2} and Corollary \ref{maincor2} under a strong assumption that the divisor $Y_1$ is ample.
The main ingredient of \cite{AV-loc pos} and \cite{AV-loc pos2} is the (divisorial) Zariski decomposition.  Instead, in this paper, we develop a new approach based on a version of the Fujita approximation (\cite[Proposition 3.7]{lehmann-nu}) and results on the moving Seshadri constants in \cite{elmnp-restricted vol and base loci}. This leads us to overcome many technical difficulties, and so our results do not require any condition on the admissible flags $Y_\bullet$ and extend to the pseudoeffective case as well.

The organization of the paper is as follows. We start in Section \ref{prelimsec} by collecting basic facts on the asymptotic base loci, asymptotic valuations, divisorial Zariski decompositions, restricted volumes, and moving Seshadri constants. In Section \ref{oksec}, we review the construction and basic properties of limiting Okounkov bodies. In the next two sections, we study the asymptotic properties of divisors via limiting Okounkov bodies. We give the proofs of Theorem \ref{main1} and Corollary \ref{maincor1} in Section \ref{restsec}, and we calculate the divisorial Zariski decomposition via the limiting Okounkov bodies in Section \ref{divzdsec}. We then turn to the augmented base loci and moving Seshadri constants. In Section \ref{augsec}, we show Theorem \ref{main2} and Corollary \ref{maincor2}. Section \ref{movsessec} is devoted to proving Theorem \ref{main3}.

\subsection*{Acknowledgment}
We would like to thank the referee for helpful and valuable suggestions and comments.

\section{Preliminaries}\label{prelimsec}

In this section, we recall basic notions and properties which we use later on.
By a \emph{variety}, we mean an irreducible connected projective variety defined over the field $\mathbb C$ of complex numbers.
Unless otherwise stated, a \emph{divisor} means an $\R$-Cartier divisor. A divisor $D$ is \emph{pseudoeffective} if its numerical equivalence class $[D] \in N^1(X)_{\R}$ lies in the pseudoeffective cone $\ol\Eff(X)$, the closure of the cone spanned by effective divisor classes.
A divisor $D$ on a variety $X$ is \emph{big} if $[D]$ lies in the interior $\text{Big}(X)$ of $\ol\Eff(X)$.
Throughout the paper, $X$ is a smooth projective variety of dimension $n$.

\subsection{Asymptotic base loci}\label{asybasubsec}
We will define the asymptotic base loci of divisors which will be used throughout the paper.
The \emph{stable base locus} $\text{SB}(D)$ of a $\Q$-divisor $D$ on $X$ is defined as
$$
\text{SB}(D):=\bigcap_{m \geq 0} \text{Bs}(|mD|)
$$
where the intersection is taken over the positive integers $m$ such that $mD$ are $\Z$-divisors.
One can also define the stable base locus $\text{SB}(D)$ of an $\R$-divisor as
$$
\text{SB}(D):=\bigcap_{D' \geq 0} \Supp(D')
$$
where the intersection is taken over all effective divisors $D'$ such that $D' \sim_{\R} D$. If $D$ is not effective, then $\text{SB}(D)=X$.
We recall that $\text{SB}(D)$ is not a numerical property of $D$ (see \cite[Example 10.3.3]{pos}).
However, the following asymptotic base loci which are defined for $\R$-divisors $D$ depend only on the numerical class $[D]\in N^1(X)_{\R}$.

\begin{definition}
Let $D$ be a divisor on $X$.
The \emph{restricted base locus} $\bm(D)$ of $D$ is defined as
$$
\bm(D):=\bigcup_A \text{SB}(D+A)
$$
where the union is taken over all ample divisors $A$ such that $D+A$ are $\Q$-divisors.
The \emph{augmented base locus} $\bp(D)$ is defined as
$$
\bp(D):=\bigcap_ A \text{SB}(D-A)
$$
where the intersection is taken over all ample divisors $A$ such that $D-A$ are $\Q$-divisors.
\end{definition}

We recall that $D$ is nef if and only if $\bm(D)=\emptyset$, and $D$ is ample if and only if $\bp(D)=\emptyset$. It is also easy to see that $D$ is not pseudoeffective if and only if $\bm(D)=X$, and $D$ is not big if and only if $\bp(D)=X$.
It is also well known that $\bm(D)$ and $\bp(D)$ do not contain any isolated points.
For more details on the asymptotic base loci, we refer to \cite{pos}, \cite{elmnp-asymptotic inv of base} and \cite{elmnp-restricted vol and base loci}.

\subsection{Asymptotic valuations}\label{asyvalsubsec}

Let $\sigma$ be a divisorial valuation of $X$, and $V:=\Cent_X \sigma$ be its center on $X$.
If $D$ is a big divisor on $X$,  we define \emph{the asymptotic valuation} of $\sigma$ at $D$ as
$$
\ord_V(||D||):=\inf\{\sigma(D')\mid D\equiv D'\geq 0\}.
$$
If $D$ is only a pseudoeffective divisor on $X$, we define
$$
\ord_V(||D||):=\lim_{\epsilon \to 0+}\ord_V(||D+\epsilon A||)
$$
for some ample divisor $A$ on $X$.
This definition is independent of the choice of $A$, and the number $\ord_V(||D||)$ depends only on the numerical class $[D]\in N^1(X)_\R$. Note that $V\subseteq\bm(D)$ if and only if $\ord_V(||D||)>0$ (see \cite[Proposition 2.8]{elmnp-asymptotic inv of base},\cite[V.1.9 Lemma]{nakayama}).
For more details, we refer to \cite{elmnp-asymptotic inv of base} and \cite{nakayama}.

\subsection{Divisorial Zariski decompositions}\label{divzdsubsec}

Let $D$ be a pseudoeffective divisor on $X$.

\begin{definition}
The \emph{divisorial Zariski decomposition} of $D$ is the decomposition
$$
D=P+N
$$
where the \emph{negative part} $N$ of $D$ is defined as
$$
N=\sum_{\codim E=1}  \ord_E(||D||)E
$$
where the summation is over the codimension $1$ irreducible subvariety $E$ of $X$ such that $ \ord_E(||D||)>0$ and the \emph{positive part} $P$ of $D$ is defined as $P:=D-N$.
\end{definition}

It is well known that the summation for the negative part $N$ is finite and the components of $N$
are linearly independent in $N^1(X)_\R$.
Furthermore, the positive part is movable, that is, $\bm(D)$ has no divisorial components.
For more details, see \cite{B} and \cite[Chapter III]{nakayama}.

\subsection{Restricted volumes}\label{resvolsubsec}
Let $D$ be a $\Q$-divisor on $X$, and $V$ be a $v$-dimensional proper subvariety of $X$ such that $V\not\subseteq\bp(D)$.
The \emph{restricted volume} of $D$ along $V$ is defined as
$$
\vol_{X|V}(D):=\limsup_{m \to \infty} \frac{h^0(X|V,mD)}{m^v/v!}
$$
where $h^0(X|V,mD)$ is the dimension of the image of the natural restriction map $\varphi:H^0(X,\mc O_X(mD))\to H^0(V,\mc O_V(mD))$ (\cite[Definition 2.1]{elmnp-restricted vol and base loci}).
As in the case with the volume function, the restricted volume $\vol_{X|V}(D)$ depends only on the numerical class of $D$, and it extends uniquely to a continuous function
$$
\vol_{X|V} : \text{Big}^V (X) \to \R
$$
where $\text{Big}^V(X)$ is the set of all $\R$-divisor classes $\xi$ such that $V$ is not properly contained in any irreducible component of $\bp(\xi)$.
By \cite[Theorem 5.2]{elmnp-restricted vol and base loci}, if $V$ is an irreducible component of $\bp(D)$, then $\vol_{X|V}(D)=0$. When $V=X$, then we recover the original volume function: $\vol_{X|X}(D)=\vol_X(D)$ for any divisor $D$.
Thus $\vol_X(D)=0$ holds when $D$ is not big.
For more details, see \cite{elmnp-restricted vol and base loci}.

\subsection{Moving Seshadri constants}\label{movsessubsec}
We first recall the definition of the Seshadri constant of a nef divisor at a point.

\begin{definition}
Let $D$ be a nef divisor on $X$. Then the \emph{Seshadri constant} $\eps(D;x)$ of $D$ at a point $x$ on $X$ is defined as
$$
\eps(D;x):=\sup \{s \mid f^*D-sE \text{ is nef} \}
$$
where $f : \widetilde{X} \to X$ is the blow-up of $X$ at $x$ with the exceptional divisor $E$.
\end{definition}

We now let
$$
\eps'(D;x) := \inf_{C\ni x} \left\{ \frac{D\cdot C}{\mult_x C} \right\}
$$
where $\inf$ runs over all irreducible curves $C$ passing through $x$.
It is well known that when $D$ is nef, $\eps(D;x)=\eps'(D;x)$ (\cite[Proposition 5.1.5]{pos}).
Furthermore, by the Seshadri's ampleness criterion (\cite[Theorem 1.4.13]{pos}), a divisor $D$ on a variety $X$ is ample if and only if $\inf_{x\in X}\eps' (D;x)>0$.
Thus for a nef divisor $D$, the Seshadri constant $\eps(D;x)$ measures the local positivity of $D$ at $x$.
For more details, we refer to \cite[Chapter 5]{pos}.

For pseudoeffective divisors, Nakamaye (\cite{nakamaye}, see also \cite{elmnp-restricted vol and base loci}) defined the following measurement.

\begin{definition}\label{def-movSesha}
Let $D$ be a pseudoeffective divisor on $X$. If $x \not\in \bp(D)$ (which implies that $D$ is big), then the \emph{moving Seshadri constant} $\eps(||D||;x)$ of $D$ at a point $x$ on $X$ is defined as
$$
\eps(||D||;x):=\sup_{f^*D = A+E} \eps(A; x)
$$
where the $\sup$ runs over all birational morphisms $f : \widetilde{X} \to X$ with $\widetilde{X}$ smooth, that are isomorphic over a neighborhood of $x$, and decompositions $f^*D = A+E$ with an ample divisor  $A$ and an effective divisor $E$ such that $f^{-1}(x)$ is not in the support of $E$. If $x \in \bp(D)$, then we simply let $\eps(||D||;x)=0$.
\end{definition}
If $D$ is nef, then $\eps(||D||;x)=\eps(D;x)$.
Note that $\eps(||D||;x)$ depends only on the numerical class of $D$. Furthermore, by \cite[Theorem 6.2]{elmnp-restricted vol and base loci}, for any fixed point $x$ of a variety $X$, the map $D \mapsto \eps(||D||;x)$ defines a continuous function on the entire N\'{e}ron-Severi space $N^1(X)_{\R}$.
For more details, we refer to \cite{elmnp-restricted vol and base loci}.

\section{Construction and basic properties of Okounkov bodies}\label{oksec}

In this section, we first explain the construction of Okounkov bodies in \cite{lm-nobody}, \cite{KK}
and limiting Okounkov bodies in \cite{CHPW} and review some of their basic properties.
Throughout this section, we fix an \emph{admissible flag} $Y_\bullet$ on a smooth projective variety $X$ of dimension $n$, which is defined as a sequence of irreducible subvarieties $Y_i$ of $X$ such that
$$
Y_\bullet: X=Y_0\supseteq Y_1\supseteq\cdots \supseteq Y_{n-1}\supseteq Y_n=\{x\}
$$
where each $Y_i$ is of codimension $i$ in $X$ and is smooth at $x$.
We denote the $\R$-linear system of a divisor $D$ by
$$
|D|_{\R}:=\{D' \mid D\sim_\R D'\geq 0\}.
$$

Let us first consider a big divisor $D$ on $X$.
For a given admissible flag $Y_\bullet$, we define a valuation-like function
\begin{equation}\label{sp}
\nu_{Y_\bullet}:|D|_{\R}\to \R_{\geq0}^n
\end{equation}
as follows.
By possibly replacing $X$ by an open subset, we may suppose that each $Y_{i+1}$ is a Cartier divisor on $Y_i$.
For $D'\in |D|_\R$, let
$$\nu_1=\nu_1(D'):=\ord_{Y_1}(D').$$
Since $D'-\nu_1(D')Y_1$ is also effective, we can define
$$\nu_2=\nu_2(D'):=\ord_{Y_2}((D'-\nu_1Y_1)|_{Y_1}).$$
Once $\nu_i=\nu_i(D')$ is defined, we define $\nu_{i+1}=\nu_{i+1}(D')$ inductively as
$$\nu_{i+1}(D'):=\ord_{Y_{i+1}}((\cdots((D'-\nu_1Y_1)|_{Y_1}-\nu_2Y_2)|_{Y_2}-\cdots-\nu_iY_i)|_{Y_{i}}).$$
By collecting the values $\nu_i(D')$, we can define a function $\nu_{Y_\bullet}$ in (\ref{sp}) as
$$
\nu(D')=(\nu_1(D'),\nu_2(D'),\cdots,\nu_n(D')).
$$

\begin{remark}\label{rem-nu<ord}
By definition, it is easy to see that for any $D'\in|D|_\R$, we have
$$\nu_i(D')\leq \ord_{Y_i}(D').$$
\end{remark}

\begin{definition}\label{def-okbd}
The \emph{Okounkov body} $\okbd_{Y_\bullet}(D)$ of a big divisor $D$ with respect to an admissible flag $Y_\bullet$ is a closed convex subset of $\R_{\geq 0}^n$ defined as follows:
\begin{equation}
\okbd_{Y_\bullet}(D):=\text{ the closure of the convex hull of }\nu_{Y_\bullet}(|D|_{\R})\subseteq\R_{\geq0}^n.
\end{equation}
The \emph{limiting Okounkov body} $\oklim_{Y_\bullet}(D)$ of a pseudoeffective divisor $D$ with respect to an admissible flag $Y_\bullet$ is defined as
$$
\oklim_{Y_\bullet}(D):=\bigcap_{\epsilon>0} \Delta_{Y_\bullet}(D+\epsilon A)
$$
where $A$ is an ample divisor on $X$.
If $D$ is not pseudoeffective, we simply put $\oklim_{Y_\bullet}(D):=\emptyset$.
\end{definition}

It is easy to see that the limiting Okounkov body $\oklim_{Y_\bullet}(D)$ is also a closed convex subset of $\R^n$. Note also that if $D$ is big, then $\okbd_{Y_\bullet}(D)=\oklim_{Y_\bullet}(D)$ by the continuity of $\okbd_{Y_\bullet}(D)$
(\cite[Theorem B]{lm-nobody}).
For this reason, we will simply use the notation $\okbd_{Y_\bullet}(D)$ instead of $\oklim_{Y_\bullet}(D)$ when $D$ is big.

\begin{remark}\label{valuation}
Instead of working with admissible flags, one can also construct the Okounkov body $\okbd_{\nu}(D)$ of a big divisor $D$ on $X$ with respect to a valuation $\nu$ of the function field $\C(X)$ of rank $n$ as in \cite{KK}. However, by \cite[Proposition 2.8 and Theorem 2.9]{CFKLRS}, for every such a valuation $\nu$ of $\C(X)$, one can find a birational morphism $f : \widetilde{X} \to X$ and an admissible flag $Y_\bullet$ on $\widetilde{X}$ such that $\okbd_{\nu}(D)=\okbd_{Y_\bullet}(f^*D)$.
\end{remark}

\begin{lemma}\label{birokbd}
Let $D$ be a pseudoeffective divisor on $X$. Consider a birational morphism $f : \widetilde{X} \to X$ with $\widetilde{X}$ smooth and an admissible flag
$$
\widetilde{Y}_\bullet : \widetilde{X}=\widetilde{Y}_0 \supseteq \widetilde{Y}_1 \supseteq \cdots \supseteq \widetilde{Y}_{n-1} \supseteq \widetilde{Y}_n=\{ x' \}
$$
on $\widetilde{X}$. Suppose that $f$ is isomorphic over $f(x')$ and
$$
Y_\bullet:=f(\widetilde{Y}_\bullet) : X=f(\widetilde{Y}_0) \supseteq f(\widetilde{Y}_1) \supseteq \cdots \supseteq f(\widetilde{Y}_{n-1}) \supseteq f(\widetilde{Y}_n)=\{ f(x') \}
$$ is an admissible flag on $X$. Then we have
$\oklim_{\widetilde{Y}_\bullet}(f^*D) = \oklim_{Y_\bullet}(D).$
\end{lemma}

\begin{proof}
It is enough to consider for the case where $D$ is big.
In this case, the assertion follows from the construction of Okounkov bodies of big divisors and the fact that $H^0(X, \mc O_X(D'))=H^0(\widetilde{X}, \mc O_{\widetilde{X}}(f^*D'))$ for any $\Z$-divisor $D'$ on $X$.
\end{proof}

\begin{lemma}\label{smflag}
Let $D$ be a pseudoeffective divisor on $X$ and $Y_\bullet$ be an admissible flag centered at a point $x$ on $X$. Then we can take a birational morphism $f: \widetilde{X} \to X$ with $\widetilde{X}$ smooth which is isomorphic over a neighborhood of $x$ and an admissible flag $\widetilde{Y}_\bullet$ on $\widetilde{X}$ such that each $\widetilde{Y}_i$ is smooth and $\oklim_{\widetilde{Y}_\bullet}(f^*D)=\oklim_{Y_\bullet}(D)$.
\end{lemma}

\begin{proof}
Recall that each subvariety $Y_i$ from the admissible flag $Y_\bullet$ is smooth at $x$.
By successively taking embedded resolutions of singularities of $Y_{n-1}, \ldots, Y_1$ in $X$, we can take a birational morphism $f: \widetilde{X} \to X$ with $\widetilde{X}$ smooth such that $f$ is isomorphic over $x$. For $1 \leq i \leq n-1$, let $\widetilde{Y}_i$ be the strict transform of $Y_i$. Then we obtain an admissible flag on $\widetilde{X}$ as follows:
$$
\widetilde{Y}_\bullet : \widetilde{X}=\widetilde{Y}_0 \supseteq \widetilde{Y}_1 \supseteq \cdots \supseteq \widetilde{Y}_{n-1} \supseteq \widetilde{Y}_n=\{ f^{-1}(x) \}.
$$
Now the assertion follows from Lemma \ref{birokbd}.
\end{proof}

\begin{remark}\label{remk_smflag}
By the above lemma, we can assume that every subvariety $Y_i$ in the admissible flag $Y_\bullet$ is smooth.
Such smoothness assumption on $Y_\bullet$ will be useful in the induction argument used in the proofs of  Theorems \ref{thrm-B-criterion}, \ref{thrm-B+criterion}, \ref{sesthm}, and Lemma \ref{ampleorigin}.
In the situation of Lemma \ref{smflag}, note that $x \in \bm(D)$ (resp. $x \in \bp(D)$) if and only if $f^{-1}(x) \in \bm(f^*D)$ (resp. $f^{-1}(x) \in \bp(f^*D)$). Furthermore, if $D$ is nef and big, then so is $f^*D$.
\end{remark}

\begin{lemma}\label{lem-okbd sum}
Let $D, D'$ be pseudoeffective divisors on $X$. Then we have
$$
\oklim_{Y_{\bullet}}(D)+\oklim_{Y_{\bullet}}(D')\subseteq\oklim_{Y_{\bullet}}(D+D').
$$
\end{lemma}

\begin{proof}
For an ample divisor $A$ and any $\epsilon,\epsilon'>0$, it follows from the convexity of $\okbd_{Y_{\bullet}}(D)$ (cf. \cite[Proof of Corollary 4.12]{lm-nobody}, \cite[Proposition 2.32]{KK}) that
$$\okbd_{Y_{\bullet}}(D+\epsilon A)+\okbd_{Y_{\bullet}}(D'+\epsilon' A)\subseteq\okbd_{Y_{\bullet}}(D+D'+(\epsilon+\epsilon')A).$$
By taking the limits $\epsilon \to 0, \epsilon' \to 0$, we obtain the statement.
\end{proof}

\begin{remark}\label{remk_additivity}
Let $D,D'$ be pseudoeffective divisors such that $E:=D'-D$ is effective. We can easily check that if the origin of $\R^n$ is contained in $\oklim_{Y_\bullet}(E)$, then
$$
\oklim_{Y_\bullet}(D) \subseteq \oklim_{Y_\bullet}(D').
$$
This situation occurs, for example, when $Y_n =\{x\} \not\subseteq \text{SB}(E)$.
For any ample divisor $A$, we can find an effective divisor $E' \sim_{\R} E+ A$ such that $\nu_{Y_\bullet}(E')=(0, \cdots, 0)$. Thus $\oklim_{Y_\bullet}(E)$ contains the origin of $\R^n$.
Clearly, the inclusion does not hold in general.
\end{remark}

The following lemma will be helpful in computing the limiting Okounkov bodies using the divisorial Zariski decompositions.

\begin{lemma}\label{okbdzdlem}
Let $D$ be a pseudoeffective divisor on $X$, and $D=P+N$ be the divisorial Zariski decomposition. Fix an admissible flag $Y_\bullet$ on $X$. Then we have $\oklim_{Y_\bullet}(D)=\oklim_{Y_\bullet}(P) + \oklim_{Y_\bullet}(N)$. In particular, if $Y_n=\{ x \} \not\subseteq \Supp(N)$, then $\oklim_{Y_\bullet}(D)=\oklim_{Y_\bullet}(P)$.
\end{lemma}

\begin{proof}
When $D$ is big, the assertion is exactly the same as \cite[Theorem C (3)]{AV-loc pos2}. If $D$ is only pseudoeffective, then the assertion follows from the big case and the definition of limiting Okounkov bodies.
\end{proof}

It is sometimes useful to work in the following restricted situations.
For $1 \leq k \leq n$, we define the \emph{$k$-th partial flag} $Y_{k\bullet}$ of a given admissible flag $Y_\bullet$ as
$$
Y_{k\bullet}:=Y_k \supseteq  \cdots\supseteq Y_n.
$$
Suppose that $D$ is a big divisor such that $Y_k\not\subseteq\bp(D)$.
We define $\nu_{Y_{k\bullet}}:|D|_\R\to\R^n_{\geq0}$ as the function $\nu_\bullet$ defined above (\ref{sp})
by letting  $\nu_i(D')=0$ for all $D'\in|D|_\R$ if $i\leq k$.

\begin{definition}\label{def-restricted okbd}
The \emph{$k$-th restricted Okounkov body} $\okbd_{Y_{k\bullet}}(D)$ of $D$ with respect to the partial flag $Y_{k\bullet}$ is defined as
the following subset of $\{ 0 \}^k \times \R_{\geq 0}^{n-k}$  (which we often identify with $\R_{\geq 0}^{n-k}$):
\begin{equation}
\okbd_{Y_{k\bullet}}(D):=\text{ the closure of the convex hull of }\nu_{{k\bullet}}(|D|_{\R}).
\end{equation}
\end{definition}

By convention, the $0$-th restricted Okounkov body is the usual Okounkov body, i.e., $\okbd_{Y_\bullet}(D)=\okbd_{Y_{0\bullet}}(D)$.
The $k$-th restricted Okounkov body $\okbd_{Y_k\bullet}(D)$ of a $\Z$-divisor $D$ is nothing but the Okounkov body of a graded linear series $W_\bullet$ in \cite[p.804]{lm-nobody}, where $W_m = \text{Im}\left(H^0(X, \mc O_{X}(mD)) \to H^0(Y_k, \mc O_{Y_k}(mD|_{Y_k})) \right)$.

The following is one of the most important properties of Okounkov bodies.

\begin{theorem}[{\cite[(2,7) p.804]{lm-nobody}}]\label{restokbd}
Let $D$ be a big divisor on $X$, and $Y_\bullet$ be an admissible flag on $X$. Assume that $Y_k \not\subseteq \bp(D)$. Then we have
$$
\vol_{\R^{n-k}}(\Delta_{Y_{k\bullet}}(D)) = \frac{1}{(n-k)!} \vol_{X|Y_k}(D).
$$
\end{theorem}

\begin{remark}\label{rem-restric okbd}
Assume that $Y_k$ is smooth. We can regard the $k$-th partial flag $Y_{k \bullet}$ as an admissible flag on $Y_k$ so that $\okbd_{Y_{k\bullet}}(D|_{Y_k})$ is a subset of $\R^{n-k}$. One can also consider $\okbd_{Y_{k\bullet}}(D) \subset \{0\}^k \times \R^{n-k}$ as an object in $\R^{n-k}$.
Then
$$\okbd_{Y_{k\bullet}}(D) \subseteq \okbd_{Y_{k\bullet}}(D|_{Y_k})$$
holds in general.
If $D$ is nef and big and $Y_k\not\subseteq\bp(D)$, then the equality holds and $\vol_{\R^{n-k}}(\okbd_{Y_{k\bullet}}(D|_{Y_k})) = \frac{1}{(n-k)!} \vol_{X}(D|_{Y_k})$.
\end{remark}

\begin{definition}
Let $\mathbf x\in\okbd_{Y_{k\bullet}}(D)$ be a point.
If $\mathbf x$ is of the form $\mathbf x=\nu_{Y_{k\bullet}}(D')$  of for some $D'\in|D|_\R$,
then $\mathbf x$ is called a \emph{($k$-th) valuative point} of $\okbd_{Y_{k\bullet}}(D)$. We denote by
$$
\Gamma_{k}:=\{\nu_{Y_{k\bullet}}(D') \mid D'\in|D|_\R\} \subseteq\{0\}^{k}\times\R^{n-k}
$$
the set of $k$-th valuative points.
\end{definition}

It is known that $\Gamma_k$ forms a dense subset in $\Delta_{Y_{k\bullet}}(D)$ (see \cite[Remark 2.12]{CFKLRS}).
Thus it is enough to take the closure of the image $\nu_{Y_{k\bullet}}(|D|_\R)$ in Definition \ref{def-okbd} and \ref{def-restricted okbd} to obtain $\okbd_{Y_{k\bullet}}(D)$.

For any subset $\Delta\subseteq\R^n$, we denote
$\Delta_{x_1=\cdots=x_k=0}:=\Delta\cap(\{0\}^k\times\R^{n-k})$. We often regard $\Delta_{x_1=\cdots=x_k=0}$ as an object in $\R^{n-k}$.
It is easy to check that the condition $Y_k\not\subseteq\bp(D)$ implies that $\okbd_{Y_{\bullet}}(D)_{x_1=\cdots=x_k=0}\neq\emptyset$.
The following is a generalization of  \cite[Theorem 4.26]{lm-nobody} and \cite[Theorem 3.4]{jow}, and it will play a crucial role in studying the augmented base loci and moving Seshadri constants.

\begin{theorem}[{\cite[Theorem 1.1]{CPW2}}]\label{thrm-okbd slice}
Let $D$ be a big divisor on $X$.
Suppose that $Y_\bullet$ is an admissible flag on $X$ such that $Y_{k} \not\subseteq \bp(D)$. Then we have
$$\okbd_{Y_{\bullet}}(D)_{x_1=\cdots=x_k=0}=\okbd_{Y_{k\bullet}}(D).$$
\end{theorem}

\section{Restricted base loci via Okounkov bodies}\label{restsec}

In this section, we prove Theorem \ref{main1} and Corollary \ref{maincor1}. More specifically, we extract the restricted base locus $\bm(D)$ of a pseudoeffective divisor $D$ from its associated limiting Okounkov bodies.
Throughout this section, $X$ is a smooth projective variety of dimension $n$.

We show the following lemma first.

\begin{lemma}\label{lem-main lemma}
Let $D$ be a big divisor on $X$, and fix an admissible flag $Y_\bullet$ on $X$.
\begin{enumerate}
\item[$(1)$] If $Y_1\subseteq\bm(D)$, then the origin of $\R^n$ is not contained in $\okbd_{Y_\bullet}(D)$.

\item[$(2)$] If $Y_1\not\subseteq\bm(D)$, then for any integer $k\geq 1$ with $Y_k\not\subseteq\bm(D)$, we have
$\okbd_{Y_\bullet}(D)_{x_1=\cdots=x_k=0}\neq \emptyset.$
\end{enumerate}
\end{lemma}

\begin{proof}
$(1)$ In this case, $Y_1$ is an irreducible component of $\bm(D)$ and $\ord_{Y_1}(||D||)>0$ by \cite[Proposition 2.8]{elmnp-restricted vol and base loci}.
Thus for any $D'\in|D|_{\R}$, we have $\nu_1(D')=\ord_{Y_1}D'\geq\ord_{Y_1}(||D||)>0$.
It follows that
$$
\inf\{x_1 \mid (x_1,\cdots,x_n)\in\okbd_{Y_\bullet}(D)\}>0.
$$
In particular, the origin of $\R^n$ is not contained in $\okbd_{Y_\bullet}(D)$.

\noindent$(2)$ Let $k\geq 1$ be an integer such that $Y_k\not\subseteq\bm(D)$.
For any $\epsilon >0$, there exists an effective divisor $D'\sim_\R D$ such that
$\ord_{Y_k}(D')<\epsilon $.
Since $Y_i\supseteq Y_k$ for $i\leq k$, we have $\ord_{Y_i}(D')\leq \ord_{Y_k}(D')$, and hence,
$\ord_{Y_i}(D')<\epsilon$ for all $i\leq k$.
By Remark \ref{rem-nu<ord}, we see that $\nu_i(D')\leq\ord_{Y_i}(D')$.
Thus
$$\nu_i(D')<\epsilon \;\text{ for all }i\leq k.$$
This implies that for any $\epsilon>0$,
there exists a valuative point $(x_1,\cdots,x_n)$ of $\okbd_{Y_\bullet}(D)$ such that $0\leq x_i<\epsilon$ for all integers
$i$ with $1\leq i\leq k$.
Thus we obtain $\okbd_{Y_\bullet}(D)_{x_1=\cdots=x_k=0}\neq \emptyset$.
\end{proof}

We now prove Theorem \ref{main1} as Theorem \ref{thrm-B-criterion}.

\begin{theorem}\label{thrm-B-criterion}
Let $D$ be a pseudoeffective divisor on $X$.
Then the following are equivalent.
\begin{enumerate}
 \item[$(1)$] $x\in\bm(D)$.

 \item[$(2)$] For any admissible flag $Y_\bullet$ centered at $x$, the origin of $\R^n$ is not contained in $\oklim_{Y_\bullet}(D)$.

 \item[$(3)$] For some admissible flag $Y_\bullet$ centered at $x$, the origin of $\R^n$ is not contained in $\oklim_{Y_\bullet}(D)$.
\end{enumerate}

\end{theorem}

\begin{proof}
We first treat the case where $D$ is big.
In this case, $\oklim_{Y_\bullet}(D)=\okbd_{Y_\bullet}(D)$.

\noindent$(1)\Rightarrow(2)$: Assume that $x\in\bm(D)$ and fix an admissible flag $Y_\bullet$ centered at $x$. By Lemma \ref{smflag} (and Remark \ref{remk_smflag}), we may assume that every subvariety $Y_i$ from the admissible flag $Y_\bullet$ is smooth.
If $Y_1\subseteq\bm(D)$, then by Lemma \ref{lem-main lemma} (1), the origin of $\R^n$ is not contained in $\okbd_{Y_\bullet}(D)$.
Thus assume that $Y_1\not\subseteq\bm(D)$.
If $k\geq 1$ is the largest integer among $i$ such that $Y_i\not\subseteq\bm(D)$, then $Y_{k+1}\subseteq\bm(D)$ and
by Lemma \ref{lem-main lemma} (2),
$$
S_{k}:=\okbd_{Y_\bullet}(D)_{x_1=\cdots=x_k=0}\neq\emptyset.
$$
We claim that $\inf\{x_{k+1} \mid (0,\cdots,0,x_{k+1},\cdots,x_n)\in S_{k}\}>0$, and hence, the origin of $\R^n$ is not contained in $\okbd_{Y_\bullet}(D)$.
Since $S_{k}\neq\emptyset$ and the valuative points are dense in $\okbd_{Y_\bullet}(D)$, it follows that
for any small $\epsilon>0$, there exists an effective divisor $D'\in|D|_\R$ which defines a valuative point $\nu_{Y_\bullet}(D')=(\nu_1,\cdots,\nu_n)\in\okbd_{Y_\bullet}(D)$ such that
$0\leq \nu_i<\epsilon$ for all integers $i$ satisfying $1\leq i\leq k$.
Let $D'_1:=D'$, $D'_2:=(D'_1-\nu_1Y_1)|_{Y_1}$ and define $D'_i$ inductively as
$$
D'_i:=(D'_{i-1}-\nu_{i-1}Y_{i-1})|_{Y_{i-1}}.
$$
Then we get
$$
\begin{array}{rl}
\nu_{k+1}=\nu_{k+1}(D')&=\ord_{Y_{k+1}}((D'_k-\nu_kY_k)|_{Y_{k}})\\
&=\ord_{Y_{k+1}}(D'_k)-\nu_k\\
&=\ord_{Y_{k+1}}((D'_{k-1}-\nu_{k-1}Y_{k-1})|_{Y_{k-1}})-\nu_k\\
&=\ord_{Y_{k+1}}(D'_{k-1})-\nu_{k-1}-\nu_k\\
&=\cdots\\
&=\ord_{Y_{k+1}}(D'_1)-\nu_1-\nu_2-\cdots -\nu_k.
\end{array}
$$
We have $\ord_{Y_{k+1}}(||D||)>0$ since $Y_{k+1}\subseteq\bm(D)$.
Since $\epsilon$ can be arbitrarily small, we may suppose that $0<\epsilon<\frac{1}{2k}\ord_{Y_{k+1}}(||D||)$. Since $0\leq \nu_i<\epsilon$ for $1\leq i\leq k$,
we obtain
$$
\begin{array}{rl}
\nu_{k+1}=\nu_{k+1}(D')&=\ord_{Y_{k+1}}(D'_1)-(\nu_1+\cdots +\nu_k)\\
&>\ord_{Y_{k+1}}(D'_1)-k\cdot\frac{1}{2k}\ord_{Y_{k+1}}(||D||)\\
&\geq\ord_{Y_{k+1}}(||D||)-\frac12\ord_{Y_{k+1}}(||D||)\\
&=\frac12\ord_{Y_{k+1}}(||D||).
\end{array}
$$
Therefore, we obtain
$$
\inf\{x_{k+1} \mid (0,\cdots,0,x_{k+1},\cdots,x_n)\in S_{k}\}\geq\frac12\ord_{Y_{k+1}}(||D||)>0.
$$
This implies that the origin of $\R^n$ is not contained in $\okbd_{Y_\bullet}(D)$.

\noindent$(2)\Rightarrow(3)$: Obvious.

\noindent$(3)\Rightarrow(1)$:
Assume that there exists an admissible flag $Y_{\bullet}$ centered at $x$ such that
the origin of $\R^n$ is not contained in $\okbd_{Y_\bullet}(D)$.
To derive a contradiction, suppose that $x\not\in\bm(D)$.
By letting $k=n$ in Lemma \ref{lem-main lemma} (2), we obtain a contradiction. Thus we have proved this theorem for big divisors.

We now turn to the proof for the case where $D$ is pseudoeffective. Fix an ample divisor $A$ on $X$. First, observe that
$$
\bm(D)=\bigcup_{\epsilon >0} \bm(D+\epsilon A).
$$
Thus $x \in \bm(D)$ if and only if $x \in \bm(D+\epsilon A)$ for all sufficiently small $\epsilon >0$.
Since $\oklim_{Y_\bullet}(D)=\bigcap_{\epsilon >0} \okbd_{Y_\bullet}(D+\epsilon A)$, the assertion easily follows from the case where $D$ is big.
\end{proof}

Note that a pseudoeffective divisor $D$ is nef if and only if $\bm(D)=\emptyset$. Thus we immediately obtain the following nefness criterion.

\begin{corollary}\label{cor-nefness}
Let $D$ be a  divisor on $X$.
Then the following are equivalent.
\begin{enumerate}
 \item[$(1)$] $D$ is nef.

 \item[$(2)$] For any admissible flag $Y_\bullet$, the origin of $\R^n$ is contained in $\oklim_{Y_\bullet}(D)$.

 \item[$(3)$] For any point $x \in X$, there exists an admissible flag $Y_\bullet$ centered at $x$ such that the origin of $\R^n$ is contained in $\oklim_{Y_\bullet}(D)$.
\end{enumerate}
\end{corollary}

Next we prove the movability criterion of divisor.
Recall that a divisor $D$ is \emph{movable} if $\bm(D)$ has no irreducible components of dimension $n-1$.

\begin{theorem}\label{thrm-movability}
Let $D$ be a pseudoeffective divisor on $X$.
The following are equivalent:
\begin{enumerate}
\item[$(1)$] $D$ is movable.
\item[$(2)$] For any admissible flag $Y_\bullet$ on $X$, we have  $\oklim_{Y_\bullet}(D)_{x_1=0}\neq\emptyset$.
\end{enumerate}
\end{theorem}

\begin{proof}
As in the proof of Theorem \ref{thrm-B-criterion}, we prove the case when $D$ is big.
The pseudoeffective case is easy and left to the readers.

\noindent$(1)\Rightarrow(2)$:
Fix an arbitrary admissible flag $Y_\bullet$ of $X$.
If $D$ is movable, then $Y_1\not\subseteq\bm(D)$.
By Lemma \ref{lem-main lemma} (2), we have
$\okbd_{Y_\bullet}(D)_{x_1=0}\neq\emptyset$.

\noindent$(2)\Rightarrow (1)$
Suppose that $D$ is not movable. Then there exists a prime divisor $E\subseteq\bm(D)$.
Take an admissible flag $Y_\bullet$ such that $Y_1=E$.
Then for any effective divisor $D'$ such that $D'\sim_\R D$,
we have
$$
\nu_1(D')=\ord_{Y_1}(D')\geq\ord_{Y_1}(||D||)>0.
$$
Thus $\okbd_{Y_\bullet}(D)_{x_1=0}=\emptyset$, which is a contradiction.
\end{proof}

\begin{corollary}\label{cor-nefness in dim2}
If $X$ is a surface, then in addition to the conditions $(1)$,$(2)$, and $(3)$ in Corollary \ref{cor-nefness},
the following condition is also equivalent:
\begin{enumerate}
\item[$(4)$] The \Obody $\oklim_{Y_\bullet}(D)$ with respect to any admissible flag $Y_\bullet$ intersects the $x_2$-axis of the plane $\R^2$.
\end{enumerate}
\end{corollary}

\begin{proof}
The condition $(4)$ is the movability condition for divisors on a surface.
On a surface, a divisor is movable if and only if it is nef.
\end{proof}

\section{Divisorial Zariski decompositions via Okounkov bodies}\label{divzdsec}
In this section, we compute the divisorial Zariski decomposition of a big divisor using the limiting Okounkov bodies.
The Zariski decomposition plays a crucial role in computing the Okounkov body of a big divisor in the surface case (see \cite[Theorem 6.4]{lm-nobody}).
As before, $X$ is a smooth projective variety of dimension $n$.

We first define the following set for a pseudoeffective divisor $D$ on $X$:
$$
div(\oklim(D)):=\left\{E ~~\left |
\begin{array}{l}
\text{ $Y_\bullet$ is an admissible flag with $Y_1=E$}\\
\text{ such that } \oklim_{Y_\bullet}(D)_{x_1=0}=\emptyset.
\end{array} \right.\right\}.
$$

\begin{lemma}\label{lem-n(okbd)}
The set $div\oklim(D)$ is finite, and $\#(div\oklim(D))=0$ if and only if $D$ is movable.
\end{lemma}
\begin{proof}
Note that for an admissible flag $Y_\bullet$ with $Y_1=E$ on $X$, we have
$$\sigma_E(||D||)=\inf\{x_1 \mid (x_1,\cdots,x_n)\in\oklim_{Y_\bullet}(D)\}.$$
Thus $\#(div\oklim(D))$ is the number of divisorial components of $\bm(D)$, so it is finite by \cite[Corollary III.1.11]{nakayama}.
The second statement follows from Theorem \ref{thrm-movability}.
\end{proof}

We now explain how to obtain the divisorial Zariski decomposition of a pseudoeffective divisor $D$ using the limiting Okounkov bodies of $D$.
If $D$ is not movable, then by Theorem \ref{thrm-movability} there exists an admissible flag $Y_\bullet$ with $Y_1=E_1$ such that $\oklim_{Y_\bullet}(D)_{x_1=0}=\emptyset$.
For such an admissible flag $Y_\bullet$, define a positive number
$$a_1:=\inf\{x_1\geq 0 \mid (x_1,\cdots,x_n)\in\oklim_{Y_\bullet}(D)\} >0$$
and consider the divisor $D-a_1E_1$.
Note that $a_1=\ord_{E_1}(||D||)$.

\begin{lemma}\label{lem-translation}
Let $D$ be a pseudoeffective divisor on $X$. If $Y_\bullet$ is an admissible flag on $X$ such that
$Y_1=E_1$ is a divisorial component of $\bm(D)$, then
$$
\oklim_{Y_\bullet}(D-a_1E_1)=\oklim_{Y_\bullet}(D)-(a_1,0,\cdots,0).
$$
\end{lemma}

\begin{proof}
It is straightforward.
\end{proof}

By Lemma \ref{lem-translation}, it is easy to see that $\oklim_{Y_\bullet}(D-a_1E_1)_{x_1=0}\neq\emptyset$ since
$$\inf\{x_1\geq 0 \mid (x_1,\cdots,x_n)\in\okbd_{Y_\bullet}(D-a_1E_1)\}=\ord_{E_1}(||D-a_1E_1||)=0.$$
We also have
$$\#(div\oklim(D-a_1E_1))=\#(div\oklim(D))-1.$$
We can continue this process by replacing $D$ by $D-a_1E_1$.
Thus after $d=\#(div\oklim(D))$ steps, we arrive at a situation where
$$
\oklim_{Y_\bullet}(D-a_1E_1-\cdots-a_dE_d)_{x_1=0}\neq\emptyset
$$
for all admissible flags $Y_\bullet$ on $X$.
Since $\#(div\oklim(D-a_1E_1-\cdots-a_dE_d))=0$, Lemma \ref{lem-n(okbd)} implies that
$P=D-a_1E_1-\cdots-a_dE_d$ is movable. Thus we obtain the
divisorial Zariski decomposition $D=P+N$ where $N=a_1E_1+\cdots+a_dE_d$.

\section{Augmented base loci via Okounkov bodies}\label{augsec}
In this section, we prove Theorem \ref{main2} and Corollary \ref{maincor2}. More precisely, we extract the augmented base locus $\bp(D)$ of a big divisor $D$ from its associated Okounkov bodies.
Throughout this section, $X$ is a smooth projective variety of dimension $n$.

First, we prove some easy lemmas.

\begin{lemma}\label{ampleorigin}
Let $D$ be a nef and big divisor on $X$, and fix an admissible flag $Y_\bullet$ on $X$ such that $x \not\in \bp(D)$. Then $U_{\geq 0}$ is contained in $\okbd_{Y_\bullet}(D)$ for some small open neighborhood $U$ of the origin of $\R^n$.
\end{lemma}

\begin{proof}
By Lemma \ref{smflag} (and Remark \ref{remk_smflag}), we may assume that every subvariety $Y_i$ in the admissible flag $Y_\bullet$ is smooth.
We prove by induction on the dimension $\dim X=n$.
The case $n=1$ is clear since we have $\okbd_{Y_\bullet}(D)=[0, \deg D]$ (\cite[Example 1.14]{lm-nobody}) and $\deg D > 0$. Now assume that $n \geq 2$.
By Theorem \ref{thrm-okbd slice}, $\okbd_{Y_\bullet}(D)_{x_1=0} = \okbd_{Y_{1\bullet}}(D)$.  Furthermore, since $D$ is nef and big, we have $\okbd_{Y_{1\bullet}}(D) = \okbd_{Y_{1\bullet}}(D|_{Y_1})$.
Note that $D|_{Y_1}$ is nef and big and $x \not\in \bp(D|_{Y_1})$.
By the induction hypothesis, there exists a simplex  $\blacktriangle_{\lambda'}\subseteq\R^{n-1}$ of length
$\lambda'=(\lambda_2,\cdots,\lambda_n)$ with all $\lambda_i>0$ such that $\blacktriangle_{\lambda'}\subseteq\okbd_{Y_\bullet}(D)_{x_1=0}$.
On the other hand, we can find a decomposition $D=A+E$ into an ample divisor $A$ and an effective divisor $E$ such that $x \not\in \Supp(E)$. By Lemma \ref{lem-okbd sum} (and Remark \ref{remk_additivity}), we have $\okbd_{Y_\bullet}(A) \subseteq \okbd_{Y_\bullet}(D)$.
Since $A - \epsilon Y_1$ is ample for all sufficiently small $\epsilon > 0$, Corollary \ref{maincor1} implies that the origin of $\R^n$ is contained in $ \okbd_{Y_\bullet}(A-\epsilon Y_1)$. By \cite[Theorem 4.26]{lm-nobody}, we have $(\epsilon, 0, \ldots, 0) \in \okbd_{Y_\bullet}(A)$, and hence, $(\epsilon, 0, \ldots, 0) \in \okbd_{Y_\bullet}(D)$.
By the convexity of $\okbd_{Y_\bullet}(D)$, the simplex  $\blacktriangle_{\lambda} \subseteq \R^n$ of length $\lambda = (\epsilon, \lambda_2, \cdots, \lambda_n)$ is contained in $\okbd_{Y_\bullet}(D)$.
\end{proof}

\begin{lemma}\label{easylem}
Let $D$ be a big divisor on $X$, and fix an admissible flag $Y_\bullet$ centered at a point $x$ on $X$.
If there is an irreducible curve $C$ passing through $x$ such that $\frac{D\cdot C}{\mult_x C} < \epsilon$ for some $\epsilon >0$ and $C \not\subseteq Y_1$, then $C \subseteq \bm(D - \epsilon Y_1)$.
Furthermore, $(\epsilon , 0, \cdots, 0) \not\in \okbd_{Y_\bullet}(D)$.
\end{lemma}

\begin{proof}
If $D - \epsilon Y_1$ is not pseudoeffective, then $\bm(D-\epsilon Y_1)=X$, and there is nothing to prove. Thus we may assume that $D- \epsilon Y_1$ is pseudoeffective.
Note that $ \mult_x C \leq Y_1\cdot C$. Thus we obtain
$$
(D - \epsilon Y_1)\cdot C \leq D\cdot C - \epsilon \mult_x C < 0.
$$
For an ample divisor $A$ and a sufficiently small number $\delta > 0$, we still have $(D-\epsilon Y_1 + \delta A) \cdot C <0$. Thus $C \subseteq \text{SB}(D-\epsilon Y_1 + \delta A)$, and hence, $C \subseteq \bm(D - \epsilon Y_1)$.

For the last statement, we suppose that $(\epsilon, 0, \cdots, 0) \in \okbd_{Y_\bullet}(D)$.
By Lemma \ref{lem-okbd sum} (and Remark \ref{remk_additivity}) and Lemma \ref{ampleorigin}, for the same $\delta>0$ as above and any sufficiently small $\epsilon' \geq 0$, we have $(\epsilon + \epsilon', 0, \cdots, 0) \in \okbd_{Y_\bullet}(D+\delta A)$.
Recall that the valuative points are dense in the Okounkov body. Thus we can take a sequence of valuative points $\{ ( x_1^i, \cdots, x_n^i) \}_i$ of $\okbd_{Y_\bullet}(D+\delta A)$ such that $\lim_{i \to \infty} (x_1^i, \cdots, x_n^i) \to (\epsilon, 0, \cdots, 0)$ and $x_1^i \geq \epsilon$. Since $(x_1^i - \epsilon, x_2^i, \cdots, x_n^i)$ is a valuative point of $\okbd_{Y_\bullet}(D-\epsilon Y_1 + \delta A)$ for every $i$, the origin of $\R^n$ is contained in $\okbd_{Y_\bullet}(D-\epsilon Y_1 + \delta A)$. However, since $\delta > 0$ is sufficiently small, we may assume that $x \in C \subseteq \bm(D-\epsilon Y_1 + \delta A)$. Thus Theorem \ref{main1} implies that the origin of $\R^n$ is not contained in $\okbd_{Y_\bullet}(D-\epsilon Y_1 + \delta A)$, so we get a contradiction.
\end{proof}

The following is the main ingredient of the proof of Theorem \ref{main2}.

\begin{theorem}\label{thrm-movSesha}
Let $D$ be a big divisor on $X$, and $x \in X$ be a point such that $x\in\bp(D) \setminus \bm(D)$.
Fix an ample divisor $A$ on $X$ and any sufficiently small $ \epsilon > 0$. Then there exists $\delta_0>0$ such that for any $\delta$ satisfying $0\leq\delta\leq \delta_0$, there exist a birational morphism $f: \widetilde{X} \to X$ with $\widetilde{X}$ smooth which is isomorphic over a neighborhood of $x$ and an irreducible curve $C$ on $\widetilde{X}$ passing through $x':=f^{-1}(x)$ such that
$$
0 \leq \frac{P\cdot C}{\mult_{x'}C}<\epsilon
$$
where $f^*(D + \delta A) = P+N$ is the divisorial Zariski decomposition.
\end{theorem}

\begin{proof}
Fix a sufficiently small number $\epsilon > 0$, and take a sufficiently small number $\delta_0 > 0$. First, we explain how to find a birational morphism $f : \widetilde{X} \to X$ and an irreducible curve $C$ on $\widetilde{X}$, and prove the upper bound in the statement. Fix any sufficiently small number $\delta$ with $0< \delta \leq \delta_0$.
Note that $x \not\in \bp(D+\delta A)$.
As in \cite[Proposition 3.7]{lehmann-nu}, we can take a birational morphism $f : \widetilde{X} \to X$ from a smooth variety $\widetilde{X}$ which is a resolution of the base ideal $\mathfrak{b}(|\lfloor m(D+\delta A) \rfloor|)$ and the asymptotic multiplier ideal $\mathcal{J}(||m(D+ \delta A)||)$ for a sufficiently large integer $m > 0$. Then $f$ is centered in $\bp(D+\delta A)$, and hence, it is isomorphic over a neighborhood of $x$. Let $x':=f^{-1}(x)$.
Since $f$ resolves the base ideal $\mathfrak{b}(|\lfloor m(D+\delta A) \rfloor|)$, we have a decomposition $f^*(\lfloor m(D+ \delta A) \rfloor )=M'+F'$ into the base point free divisor $M'$ and the fixed part $F'$. Let $M := \frac{1}{m}M', F := \frac{1}{m}F'$ and $f^*(D+\delta A) = P+N$ be the divisorial Zariski decomposition.
For a sufficiently positive very ample $\Z$-divisor $H'$ on $X$, take an effective $\Z$-divisor $G$ such that $G \equiv \lfloor b(D + \delta A) \rfloor - (K_X + (n+1)H')$ where $b > 0$ is sufficiently large integer.
Note that $H'$ and $b$ can be chosen independently of $m$.
Since the base locus of $|G|$ is contained in $\bp(D+\delta A)$, we may assume that $x \not\in \Supp(G)$.
In \cite[Proof of Proposition 3.7]{lehmann-nu}, Lehmann actually shows that
$M \leq P \leq M + \frac{1}{m}f^*G$.
It is instructive to remark that the main ingredients of \cite[Proof of Proposition 3.7]{lehmann-nu} are Nadel vanishing theorem, Castelnuovo-Mumford regularity (see \cite[Lemma 3.9]{lehmann-nu}), and \cite[Proposition 2.5]{elmnp-asymptotic inv of base}.
We now choose a sufficiently positive ample divisor $H$ on $X$ such that $G+H$ is ample and $x \not\in \Supp(G+H)$. Then we obtain
$$
M \leq P \leq M + \frac{1}{m}f^*(G+H).
$$

Note that $\eps(||D||;x)=0$ since $x\in\bp(D)$.  Recall also that $\eps(||\cdot||;x) : N^1(X)_{\R} \to \R_{\geq 0}$ is a continuous function.
Therefore, since $m>0$ is sufficiently large and $\delta>0$ is sufficiently small, we have
$\eps \left( \left| \left| D+ \delta A + \frac{1}{m}(G+H) \right| \right|;x \right) < \epsilon$.
Since $M + \frac{1}{m}f^*(G+H) \leq f^* \left( D+ \delta A + \frac{1}{m}(G+H) \right)$, we have
$\eps \left( M + \frac{1}{m}f^*(G+H); x' \right) < \epsilon$.
Then there exists an irreducible curve $C$ on $\widetilde{X}$ passing through $x'$ such that
$$
\frac{\left( M + \frac{1}{m}f^*(G+H) \right) \cdot C}{\mult_{x'} C} < \epsilon.
$$
Note that $0 \leq M+\frac{1}{m}f^*(G+H) - P \leq \frac{1}{m}f^*(G+H)$ and $x' \not\in \Supp(f^*(G+H))$. Thus $C \not\subseteq \text{SB}\left( M+\frac{1}{m}f^*(G+H) - P \right)$, and hence, $P \cdot C \leq \left( M + \frac{1}{m}f^*(G+H) \right) \cdot C$. We then obtain
$$
\frac{P\cdot C}{\mult_{x'}C}<\epsilon.
$$
This proves the upper bound in the statement for $\delta >0$.

Now, we consider the divisorial Zariski decomposition $f^*D=P'+N'$. We successively use the above notations. By \cite[Proposition III.1.14]{nakayama}, $P' + f^*(\delta A) \leq P$ and $N' \geq N$.
Since $x \not\in \bm(D)$, it follows that $\Supp(N')$ does not contain $C$. Note that $0\leq N' - N=P - P' - f^*(\delta A)\leq N'$. Thus  $C \not\subseteq \text{SB}( P - P' - f^*(\delta A) )$. This implies that
$P' \cdot C \leq (P' +f^*(\delta A)) \cdot C \leq P \cdot C$,
and hence, we also get
$$
\frac{P'\cdot C}{\mult_{x'}C}<\epsilon.
$$
Thus we have shown the upper bound in the statement for all cases.

For the lower bound in the statement, it is sufficient to prove that $P \cdot C \geq 0$ and $P' \cdot C \geq 0$.
If $P' \cdot C < 0$, then $C \subseteq \bm(P')$, which is a contradiction to the given condition $x \not\in \bm(D)$. Thus $P' \cdot C \geq 0$.
We can also show similarly that $P \cdot C \geq 0$.
\end{proof}

\begin{remark}\label{rmrk_movSesha}
Under the notation of Theorem \ref{thrm-movSesha}, we consider any birational morphism $f' : \widetilde{X}' \to \widetilde{X}$ with $\widetilde{X}'$ smooth which is isomorphic over a neighborhood of $x'$. Let $f'^*f^*(D+\delta A)=P'+N'$ be the divisorial Zariski decomposition, and $C'$ be the strict transform of $C$ on $\widetilde{X}'$. Since the support of the effective divisor $N'-f'^*N = f'^*P - P'$  does not contain $C'$, we have
$$
0 \leq \frac{P' \cdot C'}{\mult_{f'^{-1}(x')}C'} < \epsilon.
$$
Thus we can freely take further blow-ups of $\widetilde{X}$ when we apply Theorem \ref{thrm-movSesha}.
\end{remark}

Now we prove Theorem \ref{main2} as Theorem \ref{thrm-B+criterion}.

\begin{theorem}\label{thrm-B+criterion}
Let $D$ be a pseudoeffective divisor on $X$.
Then the following are equivalent.
\begin{enumerate}
\item[$(1)$] $x \in \bp(D)$.

\item[$(2)$] For any admissible flag $Y_\bullet$ centered at $x$, the subset $U_{\geq0}$ of $\R^n_{\geq 0}$ is not contained in $\oklim_{Y_\bullet}(D)$ for any small open neighborhood $U$ of the origin of $\R^n$.

\item[$(3)$] For some admissible flag $Y_\bullet$ centered at $x$, the subset $U_{\geq0}$ of $\R^n_{\geq 0}$ is not contained in $\oklim_{Y_\bullet}(D)$ for any small open neighborhood $U$ of the origin of $\R^n$.
\end{enumerate}
\end{theorem}

\begin{proof}
Suppose first that $D$ is not big. Then $\bp(D)=X$. Furthermore, we have $\dim\oklim_{Y_\bullet}(D)<n$ and $\vol_{\R^n} (\oklim_{Y_\bullet}(D))=0$ for any admissible flag $Y_\bullet$ (see \cite{CHPW}). Thus $\oklim_{Y_\bullet}(D)$ cannot contain any $n$-dimensional convex set, and thus, there is nothing to prove. Therefore, from now on, we assume that $D$ is big and we write $\oklim_{Y_\bullet}(D)=\okbd_{Y_\bullet}(D)$.

\noindent$(1)\Rightarrow(2)$:
Assume that $x \in \bp(D)$ and fix an admissible flag $Y_\bullet$ centered at $x$.
If $x \in \bm(D)$, then this implication follows from Theorem \ref{main1}.
Thus we may assume that $x \in \bp(D) \setminus \bm(D)$ so that the origin of $\R^n$ is contained in $\okbd_{Y_\bullet}(D)$.
By Lemma \ref{smflag} (and Remark \ref{remk_smflag}), we may assume that every subvariety $Y_\bullet$ in the admissible flag $Y_\bullet$ is smooth.
To derive a contradiction, suppose that a simplex $\blacktriangle_{\lambda}\subseteq\R^n$ of length $\lambda=(\epsilon, \ldots, \epsilon)$
with some $\epsilon >0$ is contained in $\okbd_{Y_\bullet}(D)$.

First, we consider the case where $D|_{Y_k}$ is not big for some $1 \leq k \leq n-1$.
Then arguing as above, we see that $\oklim_{Y_{k\bullet}}(D|_{Y_k})$ does not contain the simplex $(\blacktriangle_{\lambda})_{x_1=\cdots=x_k=0} $.
Now, take an ample divisor $A$ and a sufficiently small number $\delta > 0$.
Since $\okbd_{Y_\bullet}(\delta A)$ contains the origin of $\R^n$ by Corollary \ref{maincor1}, it follows from Lemma \ref{lem-okbd sum} (and Remark \ref{remk_additivity}) that
$$
\blacktriangle_{\lambda} \subseteq  \okbd_{Y_\bullet}(D)  \subseteq \okbd_{Y_\bullet}(D+\delta A).
$$
Since $x \not\in \bm(D)$, it follows that $x \not\in \bp(D+\delta A)$ and $Y_k \not\subseteq \bp(D+\delta A)$. By applying Theorem \ref{thrm-okbd slice}, we see that
$$
(\blacktriangle_{\lambda})_{x_1=\cdots=x_k=0} \subseteq \okbd_{Y_\bullet}(D+\delta A)_{x_1=\cdots=x_k=0} = \okbd_{Y_{k\bullet}}(D+\delta A) \subseteq \okbd_{Y_{k\bullet}}((D+\delta A)|_{Y_k}).
$$
However, $\oklim_{Y_{k\bullet}}(D|_{Y_k}) = \bigcap_{\delta > 0} \okbd_{Y_{k\bullet}}((D+\delta A)|_{Y_k})$ does not contain the simplex $(\blacktriangle_{\lambda})_{x_1=\cdots=x_k=0}$. This is a contradiction.

It remains to consider the case where $D|_{Y_k}$ is big for all $k$ such that $1 \leq k \leq n-1$. Fix an ample divisor $A$ on $X$.
Then by Theorem \ref{thrm-movSesha}, for a sufficiently small number $\delta > 0$,
there exist a birational morphism $f : \widetilde{X} \to X$ with $\widetilde{X}$ smooth which is isomorphic over a neighborhood of $x$ and an irreducible curve $C$ on $\widetilde{X}$ passing through $x':=f^{-1}(x)$ such that
$$
0 \leq \frac{P\cdot C}{\mult_{x'}C}<\epsilon
$$
where $f^*(D+\delta A)=P+N$ is the divisorial Zariski decomposition.
We take an admissible flag $\widetilde{Y}_\bullet$ on $\widetilde{X}$ by taking strict the transforms $\widetilde{Y}_i$ of $Y_i$ on $\widetilde{X}$.
Possibly by taking further blow-ups, we can assume that every subvariety $\widetilde{Y}_i$ is smooth.
By Lemmas \ref{birokbd} and \ref{okbdzdlem} and the above argument, we obtain
$$
\blacktriangle_{\lambda} \subseteq \okbd_{Y_\bullet}(D)=\okbd_{\widetilde{Y}_\bullet}(f^*D)=\okbd_{\widetilde{Y}_\bullet}(P).
$$
Since $x \not\in \bp(D+\delta A)$ and $Y_k \not\subseteq \bp(D+\delta A)$, we have $x' \not\in \bp(P)$ and $\widetilde{Y}_k \not\subseteq \bp(P)$. By applying Theorem \ref{thrm-okbd slice}, we see that
$$
(\blacktriangle_{\lambda})_{x_1=\cdots=x_k=0}\subseteq \okbd_{\widetilde{Y}_\bullet}(P)_{x_1=\cdots=x_k=0} = \okbd_{\widetilde{Y}_{k\bullet}}(P) \subseteq \okbd_{\widetilde{Y}_{k\bullet}}(P|_{\widetilde{Y}_k})
$$
for $0 \leq k \leq n-1$.

Suppose that there exists an integer $k$ with $0 \leq k \leq n-2$ such that $C \subseteq \widetilde{Y}_k$ and $C \not\subseteq \widetilde{Y}_{k+1}$. We have
$$
\frac{P\cdot C}{\mult_{x'}C}=\frac{P|_{\widetilde{Y}_k}\cdot
C}{\mult_{x'}C}<\epsilon .
$$
By Lemma \ref{easylem}, $(\epsilon, 0, \cdots, 0) \not\in \okbd_{\widetilde{Y}_{k\bullet}}(P|_{\widetilde{Y}_k})$ holds in $\R^{n-k}$. Since $(\blacktriangle_{\lambda})_{x_1=\cdots=x_k=0} \subseteq \okbd_{\widetilde{Y}_{k\bullet}}(P|_{\widetilde{Y}_k})$, we get a contradiction.
Therefore, we must have $C = \widetilde{Y}_{n-1}$. In this case, we have $\mult_{x'}C=1$, so $0 \leq P\cdot C < \epsilon$. Let $P|_{\widetilde{Y}_{n-2}}=P'+N'$ be the Zariski decomposition on the smooth surface $\widetilde{Y}_{n-2}$.
If $x' \in \Supp(N')$, then the origin of $\R^2$ is not contained in $\okbd_{\widetilde{Y}_{n-2 \bullet}}(P|_{\widetilde{Y}_{n-2}})$ by Theorem \ref{main1}. This gives a contradiction to that $(\blacktriangle_{\lambda})_{x_1=\cdots=x_{n-2}=0} \subseteq \okbd_{\widetilde{Y}_{n-2 \bullet}}(P|_{\widetilde{Y}_{n-2}})$. Thus we assume that $x' \not\in \Supp(N')$. Then we have
$\okbd_{\widetilde{Y}_{n-2 \bullet}}(P|_{\widetilde{Y}_{n-2}}) = \okbd_{\widetilde{Y}_{n-2 \bullet}}(P')$.
Since $C \not\subseteq \Supp(N')$, we have
$$
P' \cdot C \leq P' \cdot C + N' \cdot C =  P|_{\widetilde{Y}_{n-2}} \cdot C = P \cdot C < \epsilon .
$$
By \cite[Theorem 6.4]{lm-nobody}, $(0, \epsilon) \not\in \okbd_{\widetilde{Y}_{n-2\bullet}}(P|_{\widetilde{Y}_{n-2}})$ in $\R^2$, which is again a contradiction. Hence we complete the proof for the implication $(1)\Rightarrow(2)$.

\noindent$(2)\Rightarrow(3)$: Obvious.

\noindent$(3)\Rightarrow(1)$:
Suppose that $x \not\in \bp(D)$.
For some sufficiently small ample divisor $A$ on $X$, we have $\bp(D)=\bm(D- A)$. By Theorem \ref{main1}, for any admissible flag $Y_\bullet$ centered at $x$, the origin of $\R^n$ is contained in $\okbd_{Y_\bullet}(D-A)$.
By Lemma \ref{ampleorigin}, the Okounkov body $\okbd_{Y_\bullet}(A)$ contains $U_{\geq 0}$ for some open neighborhood $U$ of the origin of $\R^n$.
By Lemma \ref{lem-okbd sum}, we have
$$
U_{\geq 0}\subseteq \okbd_{Y_\bullet}(D-A)+\okbd_{Y_\bullet}( A)\subseteq\okbd_{Y_\bullet}(D).
$$
Therefore, $U_{\geq0}$ is also contained in $\okbd_{Y_\bullet}(D)$.
\end{proof}

Recall that a divisor $D$ is ample if and only if $\bp(D)=\emptyset$. Thus we immediately obtain the following ampleness criterion of divisors.

\begin{corollary}\label{cor-ampleness}
Let $D$ be a divisor on $X$. Then the following are equivalent.
\begin{enumerate}
 \item[$(1)$] $D$ is ample.

 \item[$(2)$] For any admissible flag $Y_\bullet$, $U_{\geq0}$ is contained in $\okbd_{Y_\bullet}(D)$ for some small open neighborhood $U$ of the origin of $\R^n$.

 \item[$(3)$] For any point $x \in X$, there exists an admissible flag $Y_\bullet$ centered at $x$ such that  $U_{\geq0}$ is contained in $\okbd_{Y_\bullet}(D)$ for some small open neighborhood $U$ of the origin of $\R^n$.
\end{enumerate}
\end{corollary}

\section{Bounds for moving Seshadri constants via Okounkov bodies}\label{movsessec}

In this section, we prove Theorem \ref{main3}. Throughout this section, $X$ is a smooth projective variety of dimension $n$. Since every assertion in this section is trivial for the curve case, we assume that $n \geq 2$.

For a convex subset $\okbd\subseteq \R^n$ containing the origin of $\R^n$, we define the \emph{maximal sub-simplex} of $\okbd$ as the simplex $\blacktriangle_\lambda$ of length $\lambda=(\lambda_1, \cdots, \lambda_n)$ where
$\lambda_i:=\max\{x_i \mid (0,\cdots,0,x_i,0,\cdots,0)\in\okbd\}$ is the $i$-th maximal length of $\okbd$ for each $i$. Note that we may have $\lambda_i=0$ for some $i$.
If the origin of $\R^n$ is not contained in $\okbd$, then we let its maximal sub-simplex be the empty set and $\lambda_i=0$ for all $i$.
If $\okbd=\oklim_{Y_\bullet}(D)$ for a pseudoeffective divisor $D$ on $X$  and $Y_\bullet$ is an admissible flag centered at $x\in X$, then each $i$-th maximal length $\lambda_i$ of $\okbd_{Y_\bullet}(D)$ depends on $D$, $x$, and $Y_\bullet$. Thus we can write $\lambda_i=\lambda_i(D;x,Y_\bullet)$.

We first compute the bounds for the Seshadri constant of nef and big divisors.

\begin{theorem}\label{sesthm}
Let $D$ be a nef and big divisor on $X$, and $x \in X$ be a point such that  $x \not\in \bp(D)$.
Fix an admissible flag $Y_\bullet$ centered at $x$.
Let $\blacktriangle_{\lambda}$ be the maximal sub-simplex of $\okbd_{Y_\bullet}(D)$ of length $\lambda = (\lambda_1, \cdots, \lambda_n)$ where $\lambda_i=\lambda_i(D;x,Y_\bullet)$.
Then we have
$$
\lambda_{\min} \leq \eps(D; x) \leq \lambda_n
$$
where $\lambda_{\min}:=\min_{1 \leq i \leq n} \{ \lambda_i \}$.
\end{theorem}

\begin{proof}
By Lemma \ref{smflag} (and Remark \ref{remk_smflag}), we can assume that all subvarieties $Y_i$ in the admissible flag  $Y_\bullet$ are smooth.
By Theorem \ref{main2}, we have $\lambda_i > 0$ for every $i$. Since $Y_n=\{x\} \not\subseteq \bp(D)$, it follows that $Y_{k} \not\subseteq \bp(D)$ for all $1 \leq k \leq n-1$.
Thus $D|_{Y_k}$ is nef and big for all $0 \leq k \leq n-1$, and $\okbd_{Y_{k\bullet}}(D) = \okbd_{Y_{k\bullet}}(D|_{Y_k})$.
By Theorem \ref{thrm-okbd slice}, we have
$$
\okbd_{Y_\bullet}(D)_{x_1=\cdots=x_k=0} = \okbd_{Y_{k\bullet}}(D) = \okbd_{Y_{k\bullet}}(D|_{Y_k}).
$$
We then see that $(\blacktriangle_{\lambda})_{x_1=\cdots=x_k=0}$ is the maximal sub-simplex of $\okbd_{Y_{k\bullet}}(D|_{Y_k})$.

First, we show the upper bound.
Recall that
$$
\eps(D;x) = \inf_{C} \left\{\frac{D\cdot C}{\mult_x C} \right\}
$$
where $\inf$ runs over all irreducible curves $C$ passing through $x$.
We have $Y_{n-1} \not\subseteq \bp(D)$. It follows from Theorem \ref{restokbd} and the above observation that
$$
\okbd_{Y_\bullet}(D)_{x_1=\cdots=x_{n-1}=0} = \{ (0,  \ldots, 0, x_n) \mid 0 \leq x_n \leq  \vol_{X|Y_{n-1}}(D)   \}.
$$
Note that $\lambda_n= \vol_{X|Y_{n-1}}(D)= D\cdot Y_{n-1}$ and $\mult_x Y_{n-1}=1$.
Thus we have
$$
\eps(D;x)  = \inf_{C} \left\{ \frac{D\cdot C}{\mult_x C} \right\} \leq \frac{D\cdot Y_{n-1}}{\mult_x Y_{n-1}} = \lambda_n.
$$

For the lower bound, it suffice to prove that
$$
\lambda_{\min}\leq\frac{D\cdot C}{\mult_x C}
$$
for any irreducible curve $C$ passing through $x$.
Note that $(\lambda_{\min}, 0, \cdots, 0) \in \okbd_{Y_\bullet}(D)$.
If $C \not\subseteq Y_1$, then it follows from Lemma \ref{easylem} that $\lambda_{\min}\leq\frac{D\cdot C}{\mult_x C}$.
When $C \subseteq Y_1$, we use the induction on the dimension $n(\geq 2)$ of $X$.
If $n=2$, then $C = Y_1$. Since $\mult_x C = 1$, we have
$$
\lambda_{\min}\leq  \lambda_2=\vol_{X|C}(D)=D\cdot C=\frac{D\cdot C}{\mult_x C}.
$$
Now we suppose that $n \geq 3$. In this case, by induction, we obtain
$$
\lambda_{\min}\leq \min_{2 \leq i \leq n} \{ \lambda_i \} \leq \frac{D|_{Y_1}\cdot C}{\mult_xC}=\frac{D\cdot C}{\mult_xC}.
$$
Hence, for any irreducible curve $C$ passing through $x$, we have $\lambda_{\min}\leq \frac{D\cdot C}{\mult_x C}$ as desired.
\end{proof}

We prove Theorem \ref{main3} as Theorem \ref{movsesthm}.

\begin{theorem}\label{movsesthm}
Let $D$ be a  pseudoeffective divisor on $X$, and $x$ be a point on $X$.
Then we have
$$
\sup_{Y_\bullet}\{\lambda_{\min}(D;x,Y_\bullet)\} \leq \eps(||D||;x) \leq \inf_{Y_\bullet}\{\lambda_n(D;x,Y_\bullet)\}
$$
where $\sup$ and $\inf$ are taken over the admissible flags $Y_\bullet$ centered at $x$.
\end{theorem}

\begin{proof}
It is enough to show that
$$
\lambda_{\min}(D;x,Y_\bullet) \leq \eps(||D||;x) \leq \lambda_{n}(D;x,Y_\bullet)
$$
for any fixed admissible flag $Y_\bullet$ on $X$ centered at $x$.

If $x \in \bp(D)$, then $\eps(||D||;x)=0$.
 If $x\in\bm(D)$, then Theorem \ref{main1} implies that the origin of $\R^n$ is not contained in $\oklim_{Y_\bullet}(D)$ and $\lambda_{\min}(D;x,Y_\bullet)=0$ by definition.
Even when $x\not\in\bm(D)$, Theorem \ref{main2} implies that $\lambda_{\min}(D;x,Y_\bullet)=0$ for any admissible flag $Y_\bullet$ centered at $x$.
Thus we assume below that $x \not \in \bp(D)$, which in particular implies that $D$ is big.

As in the proof of Theorem \ref{thrm-movSesha}, using \cite[Proposition 3.7]{lehmann-nu}, we can take an ample divisor $H$ on $X$ and a birational morphism $f:\widetilde{X} \to X$ isomorphic over a neighborhood of $x$ and a decomposition $f^*D = M+F$ into a nef and big divisor $M$ and an effective divisor $F$ such that
$$
M \leq P \leq M + \frac{1}{m}f^*H
$$
where $f^*D = P+N$ is the divisorial Zariski decomposition and $m>0$ is a sufficiently large integer.
Note that the choice of $H$ is independent of $m$ and we may assume $x \not\in \Supp(H)$.
Let $\widetilde{Y}_\bullet$ be the admissible flag on $\widetilde{X}$ obtained by taking the strict transforms of subvarieties of $Y_\bullet$, and $x':=f^{-1}(x)$. By Lemmas \ref{birokbd} and \ref{okbdzdlem}, $\okbd_{Y_\bullet}(D) =\okbd_{\widetilde{Y}_\bullet}(P)$.
Note that $P-M=F-N$ is an effective divisor and $F-N \leq F \leq f^*D$.
Since $x \not\in \bp(D)$, it follows that $x \not\in \Supp(F-N)$ and so $x \not\in \text{SB}(P-M)$. On the other hand, since $0 \leq M + \frac{1}{m}f^*H - P \leq \frac{1}{m}f^*H$ and $x \not\in \Supp(H)$, it follows that $x \not\in \text{SB}\left( M + \frac{1}{m}f^*H - P \right)$. By Lemma \ref{lem-okbd sum} (and Remark \ref{remk_additivity}), we have
$$
\okbd_{\widetilde{Y}_{\bullet}}(M) \subseteq \okbd_{\widetilde{Y}_\bullet}(P) \subseteq \okbd_{\widetilde{Y}_\bullet}\left( M + \frac{1}{m} f^*H \right).
$$
Thus we obtain
$$
\lambda_i(M;x',\widetilde{Y}_\bullet) \leq \lambda_i(D;x,Y_\bullet) \leq \lambda_i \left( M + \frac{1}{m} f^*H ; x', \widetilde{Y}_\bullet \right)
$$
for any $i$ such that  $1 \leq i \leq n$.

Note that $\eps(||D||;x) \leq \eps \left( M + \frac{1}{m}f^*H; x' \right) \leq \eps \left( ||D +  \frac{1}{m}H|| ; x\right)$.
Recall that $\eps(||\cdot||;x) : N^1(X)_{\R} \to \R_{\geq 0}$ is continuous.
For sufficiently small $\delta > 0$, by possibly taking sufficiently large $m>0$ and taking further blow-ups of $\widetilde{X}$, we have
$$
 \eps \left( M + \frac{1}{m}f^*H; x' \right) - \delta \leq \eps(||D||;x) \leq \eps(M; x') + \delta.
$$
Now we apply Theorem \ref{sesthm} to see that
$$
\eps(M; x') \leq \lambda_n(M; x', \widetilde{Y}_\bullet) ~~\text{ and }~~ \lambda_{\min}\left( M + \frac{1}{m} f^*H ; x', \widetilde{Y}_\bullet \right) \leq \eps \left(M + \frac{1}{m} f^*H; x' \right).
$$
For the upper bound, we note that
$$
\eps(||D||;x) - \delta \leq \eps(M; x') \leq  \lambda_n(M; x', \widetilde{Y}_\bullet)  \leq  \lambda_n(D;x,Y_\bullet).
$$
Since $\delta > 0$ can be chosen arbitrarily small, we get $\eps(||D||;x)  \leq  \lambda_i(D;x,Y_\bullet)$ as desired.
For the lower bound, we also note that
$$
\lambda_{\min}(D; x, Y_\bullet) \leq \lambda_{\min} \left( M + \frac{1}{m} f^*H ; x', \widetilde{Y}_\bullet \right) \leq  \eps \left(M + \frac{1}{m} f^*H; x' \right) \leq \eps(||D||;x) + \delta.
$$
Since $\delta > 0$ can be chosen arbitrarily small, we also obtain $ \lambda_{\min}(D;x,Y_\bullet) \leq \eps(||D||;x) $ as desired.
\end{proof}

\begin{example}\label{equality}
Let $m$ be a positive integer.
Let $X=\P^2$ and $D\sim L$  where $L$ is a line on $\P^2$.
Consider an admissible flag $Y_\bullet$ where $Y_1$ is a general member of $|mL|$.
Then
$$
\okbd_{Y_\bullet}(D) = \left\{(x_1, x_2) \in \R_{\geq 0}^2 \left| \; m^2x_1+x_2 \leq m \right.\right\}.
$$
Thus $\lambda_{\min}(D;x,Y_\bullet)=\frac{1}{m}$ and $\lambda_2(D;x,Y_\bullet)=m$.
Since $\eps(D;x)=1$, the inequalities in Theorem \ref{sesthm} are strict if $m>1$.
However, if $m=1$, then the equalities in Theorems \ref{sesthm} and \ref{movsesthm} hold.
\end{example}

\begin{example}\label{strineqrem}
As in \cite[Remark 4.9]{AV-loc pos}, we also consider a fake projective plane $S$ such that $K_S=3H$ where $\Pic(S)=\Z \cdot [H]$ and $H^2=1$ (see \cite[10.4]{PS} for the existence of such a surface).
Fix an admissible flag $Y_\bullet : Y_0=S \supseteq Y_1=C \supseteq Y_0=\{x\}$ where $x$ is a very general point on $S$.
Note that $H^0(S, K_S)=0$ and so $H^0(S, H)=0$. Thus $C \in |kH|$ for some integer $k>1$.
As in the previous example, we have
$$
\okbd_{Y_\bullet}(H) = \left\{(x_1, x_2) \in \R_{\geq 0}^2 \left| \; k^2x_1+x_2 \leq k  \right.\right\}.
$$
Note that $\eps(H;x)=1$. However, since we always have $k>1$, both inequalities in Theorem \ref{movsesthm} are strict.
\end{example}

We finally present an application of Theorem \ref{main3}.
Let $A$ be an ample $\Z$-divisor on $X$. We may regard $X$ as a compact complex manifold, and we can choose a K\"{a}hler from $\omega_A$ representing $c_1(A) \in H^2(X, \C)$. Then $(X, \omega_A)$ is a symplectic manifold of real dimension $2n$. The \emph{Gromov width} $w_G(X, \omega_A)$ is the supremum of all $\lambda > 0$ for which there exists a $\mathcal{C}^{\infty}$ embedding $j: B(\lambda) \hookrightarrow X$ of the open ball $B(\lambda)$ of radius $\lambda$ with the standard symplectic form $\omega_{\text{std}}$ into $X$ such that $j^*\omega_A = \omega_{\text{std}}$. In \cite{MP} (see also \cite[Theorem 5.1.22]{pos}), McDuff and Polterovich prove that
$$
\omega_G(X, \omega_A) \geq \sqrt{\frac{\sup_{x \in X} \eps(A;x)}{\pi}}.
$$
As an immediate corollary of Theorem \ref{main3}, we obtain the following, which was independently shown by Kaveh using a different method (\cite[Corollary 12.4]{K}).

\begin{corollary}\label{kaveh}
Let $A$ be an ample $\Z$-divisor on $X$. Then we have
$$
\omega_G(X, \omega_A) \geq \sqrt{\frac{\sup_{x \in X} \sup_{Y_\bullet}\{ \xi_{\min}(A;x,Y_\bullet) \}}{\pi}}
$$
where the right $\sup$ is taken over the admissible flags $Y_\bullet$ centered at $x$.
\end{corollary}

\begin{remark}\label{gromov}
In \cite[Theorem 1.1 and Corollary 1.2]{FLP}, Fang-Littelmann-Pabiniak compute the Gromov widths of coadjoint orbits of all compact connected simple Lie groups. Recall that such a coadjoint orbit is symplectomorphic to a flag variety $(X=G/P, \omega_A)$ where $A$ is a very ample divisor.
They actually show that there exists a certain Okounkov body $\okbd_{\nu}(A)$ with respect to a valuation $\nu$ of $\C(X)$ such that the minimum $\lambda_{\min}$ of maximal lengths $\lambda_i$ of $\okbd_{\nu}(A)$ computes the Gromov width.
Let $f : \widetilde{X} \to X$ be a birational morphism such that there exists an admissible flag $Y_\bullet$ on $\widetilde{X}$ with $\okbd_{\nu}(f^*A)=\okbd_{Y_\bullet}(f^*A)$ (see Remark \ref{valuation}).
Then we see that the Okounkov body $\okbd_{Y_\bullet}(f^*A)$ computes $\sup_{x \in \widetilde{X}} \eps(f^*A; x)$.
\end{remark}

\end{document}